\documentclass[12pt]{amsproc}
\usepackage{amscd,amssymb,graphicx,xy}
\xyoption{all}

\newtheorem{theorem}{Theorem}[section]
\newtheorem{lemma}[theorem]{Lemma}
\newtheorem{proposition}[theorem]{Proposition}

\theoremstyle{definition}
\newtheorem{definition}[theorem]{Definition}
\newtheorem{notation}[theorem]{Notation}

\numberwithin{equation}{section}

\DeclareMathOperator{\comp}{\#}
\renewcommand{\d}{\partial}
\newcommand{\e}{\epsilon}
\DeclareMathOperator{\Hom}{Hom}
\DeclareMathOperator{\id}{id}
\DeclareMathOperator{\im}{im}
\DeclareMathOperator{\Or}{\mathcal{O}}
\renewcommand{\v}{\vee}
\newcommand{\w}{\wedge}

\begin{document}

\title{Complicial structures in the nerves of omega-categories}

\author{Richard Steiner}

\address{School of Mathematics and Statistics, University of
Glasgow, University Gardens, Glasgow, Great Britain G12 8QW}

\email{Richard.Steiner@glasgow.ac.uk}



\begin{abstract}
It is known that strict omega-categories are equivalent through the nerve functor to complicial sets and to sets with complicial identities. It follows that complicial sets are equivalent to sets with complicial identities. We discuss these equivalences. In particular we give a conceptual proof that the nerves of omega-categories are complicial sets, and a direct proof that complicial sets are sets with complicial identities.
\end{abstract}

\maketitle

\section{Introduction} \label{S1}

This paper is concerned with the nerves of strict $\omega$-categories (all $\omega$-categories in this paper will be strict $\omega$-categories). The nerve functor, which was defined by Street in~\cite{B7}, takes $\omega$-categories to simplicial sets. The nerve functor is in fact an equivalence from the category of $\omega$-categories to a category of simplicial sets with additional structure, and this additional structure has been described in two different ways. Verity has shown in~\cite{B9} that $\omega$-categories are equivalent to complicial sets; in~\cite{B5} I have shown that $\omega$-categories are equivalent to sets with complicial identities. Here a complicial set is a simplicial set with a distinguished class of `thin' elements, and a set with complicial identities is a simplicial set with additional partial binary operations; in both cases the extra structure is subject to an appropriate list of axioms. The axioms for complicial sets are simple and elegant, but their calculational implications are not very obvious; the axioms for sets with complicial identities are algebraic, in fact equational, like the axioms for cubical nerves in~\cite{B1}, and they allow one to make concrete calculations as in~\cite{B4}.

It follows from a combination of \cite{B5}~and~\cite{B9} that the categories of strict $\omega$-categories, of complicial sets and of sets with complicial identities are all equivalent. The object of this paper is to shed some light on these equivalences. There are two parts. In the first part we give a conceptual proof that the nerve of an $\omega$-category is a complicial set (Theorem~\ref{T8.2}). This result was first proved computationally by Street in~\cite{B8} (actually in a rather stronger form), but the conceptual proof given here still seems to be enlightening. In the second part of the paper we show directly that a complicial set is a set with complicial identities (Theorem~\ref{T9.2}); this is straightforward, although the details are somewhat complicated. 

The theory of complicial sets is based on thin fillers for a class of `complicial' or `admissible' horns in simplicial sets with thin elements. The process of filling an $(n-1)$-dimensional complicial horn is therefore an $n$-ary operation with a rather complicated domain. The idea behind both parts of this paper is that an $(n-1)$-dimensional complicial horn in a complicial set is equivalent to a pair of arbitrary $(n-1)$-dimensional elements with a common face. This idea leads to the conceptual treatment of fillers in the first part and to the binary operations in the second part. The connection with binary operations appeared long ago in unpublished work of Street~\cite{B6}.

The structure of the paper is as follows. In Sections \ref{S2} and~\ref{S3} we recall the construction of the nerve of an $\omega$-category as a simplicial set and we show that it is a stratified simplicial set (a simplicial set with thin elements). In Section~\ref{S4} we recall the definition of complicial sets. In Section~\ref{S5} we discuss horns in the nerves of $\omega$-categories. In Section~\ref{S6} we discuss a class of horns represented by pairs of faces, which are later shown to be the complicial horns, and in Section~\ref{S7} we discuss the fillers for these horns. Section~\ref{S7} contains the bulk of the proof that the nerves of $\omega$-categories are complicial sets, and the proof is completed in Section~\ref{S8}.

The second part of the paper consists of Section~\ref{S9}. First we define sets with complicial identities; then we show that complicial sets are sets with complicial identities.

\section{Nerves as simplicial sets} \label{S2}

In this section we describe the nerve functor~$N$ from $\omega$-categories to simplicial sets. It was constructed by Street in~\cite{B7} in the following way: there is a sequence of $\omega$-categories $\Or_0,\Or_1,\ldots\,$ called \emph{orientals}, and there is a functor from the simplex category to $\omega$-categories given on objects by $[n]\mapsto\Or_n$; the nerve of an $\omega$-category~$C$ is the simplicial set $NC$ given by
$$N_n C=\Hom(\Or_n,C).$$
We will describe this construction in terms of chain complexes, following \cite{B2}~and~\cite{B3}.

We use a one-sorted description of $\omega$-categories, so that an $\omega$-category~$C$ is a set which serves as the set of morphisms for an infinite sequence of commuting category structures. The composition operations will be denoted $\comp_0,\comp_1,\ldots\,$, the left identity for~$x$ under~$\comp_n$ will be denoted $d_n^- x$, and the right identity for~$x$ under~$\comp_n$ will be denoted $d_n^+ x$. If $n<p$ then the identities for~$\comp_n$ must also be identities for~$\comp_p$; every member of~$C$ must be an identity for some operation~$\comp_n$.

When we say that the category structures commute we mean in particular that for $m\neq n$ there are equalities
$$d_m^\alpha d_n^\beta x=d_n^\beta d_m^\alpha x,\quad
d_m^\alpha(x\comp_n y)=d_m^\alpha x\comp_n d_m^\alpha y,$$
where the signs $\alpha$~and~$\beta$ are arbitrary. We deduce that the identities for any particular operation~$\comp_m$ form a sub-$\omega$-category. This gives the following result.

\begin{proposition} \label{P2.1}
Let $m$ be a nonnegative integer and let $C$ be an $\omega$-category generated by elements~$g$ such that $d_m^- g=d_m^+ g=g$. Then $d_m^- x=d_m^+ x=x$ for all $x\in C$.
\end{proposition}

Next we recall the construction of orientals in terms of chain complexes, as in \cite{B2}~and~\cite{B3}. We use a category of free augmented directed complexes defined as follows.

\begin{definition} \label{D2.2}
A \emph{free augmented directed complex} is an augmented chain complex of free abelian groups concentrated in nonnegative dimensions, together with a prescribed basis for each chain complex. A \emph{morphism of free augmented directed complexes} is an augmentation-preserving chain map which takes sums of basis elements to sums of basis elements.
\end{definition}

Note that the word sums is to be taken literally: a sum of basis elements is an expression 
$$a_1+\ldots+a_k$$
with $k\geq 0$ such that $a_1,\ldots,a_k$ are basis elements; no subtraction is allowed.

The main example used in this paper is the chain complex of the $n$-simplex. This chain complex will be denoted~$\Delta_n$. The $q$-dimensional chain group of~$\Delta_n$ has prescribed basis consisting of the $(q+1)$-tuples of integers
$$[a_0,\ldots,a_q]\quad (0\leq a_0<a_1<\ldots<a_q\leq n);$$
the boundary of $[a_0,\ldots,a_q]$ for $q>0$ is the alternating sum
$$[a_1,\ldots,a_q]-[a_0,a_2,\ldots,a_q]+\ldots+(-1)^q[a_0,\ldots,a_{q-1}];$$
the augmentation is given by $\e[a_0]=1$.

We define morphisms $\d_i^\v$~and~$\e_i^\v$ of free augmented directed complexes in the obvious way, so that they correspond contravariantly to face and degeneracy operations. For $n>0$ and for $0\leq i\leq n$ the face operation
$$\d_i^\v\colon\Delta_{n-1}\to\Delta_n$$
is given on basis elements by
$$\d_i^\v[a_0,\ldots,a_q]=[a_0',\ldots,a_q']$$
such that
$$a_j'=\begin{cases}
a_j& (0\leq a_j\leq i-1),\\
a_{j+1}& (i\leq j\leq n-1).
\end{cases}$$
For $0\leq i\leq n$ the degeneracy operation
$$\e_i^\v\colon\Delta_{n+1}\to\Delta_n$$
is given on a basis element $[a_0,\ldots,a_q]$ as follows: if the terms~$a_j$ do not include both $i$~and $i+1$ then
$$\e_i^\v[a_0,\ldots,a_q]=[a_0'',\ldots,a_q'']$$
such that
$$a_j'=\begin{cases}
a_j& (0\leq a_j\leq i),\\
a_{j-1}& (i+1\leq j\leq n+1);
\end{cases}$$
if the terms~$a_j$ do include both $i$~and $i+1$ then
$$\e_i^\v[a_0,\ldots,a_q]=0.$$
These morphisms satisfy the contravariant forms of the standard simplicial relations,
\begin{align*}
&\d_j^\v\d_i^\v=\d_i^\v\d_{j-1}^\v\quad (i<j),\\
&\e_j^\v\d_i^\v=\d_i^\v\e_{j-1}^\v\quad (i<j),\\
&\e_j^\v\d_j^\v=\e_j^\v\d_{j+1}^\v=\id,\\
&\e_j^\v\d_i^\v=\d_{i-1}^\v\e_j^\v\quad (i>j+1),\\
&\e_j^\v\e_i^\v=\e_i^\v\e_{j+1}^\v\quad (i\leq j).
\end{align*}
They therefore yield a functor from the simplex category to the category of augmented directed complexes.

It is straightforward to check that there is a functor~$\nu$ from free augmented directed complexes to $\omega$-categories defined as follows.

\begin{notation} \label{N2.3}
Let $K$ be a free augmented directed complex. Then $\nu K$ is the $\omega$-category with the following structure. The members of $\nu K$ are the infinite sequences 
$$(\,x_0^-,x_0^+\mid x_1^-,x_1^+\mid\ldots\,)$$
with finitely many non-zero terms such that $x_q^-$~and~$x_q^+$ are sums of prescribed $q$-dimensional basis elements in~$K$, such that
$$\e x_0^-=\e x_0^+ =1$$
and such that
$$\d x_q^- =\d x_q^+=x_{q-1}^+ - x_{q-1}^-\quad (q>0).$$
The operations $d_n^-$~and~$d_n^+$ are given by
$$d_n^\alpha(\,x_0^-,x_0^+\mid\ldots\,)
=(\,x_0^-,x_0^+\mid\ldots\mid x_{n-1}^-,x_{n-1}^+\mid x_n^\alpha,x_n^\alpha\mid 0,0\mid\ldots\,).$$
If $d_n^+x=d_n^- y=z$, say, then
$$x\comp_n y=x-z+y,$$
where the addition and subtraction are performed termwise.
\end{notation}

By composing the functor $[n]\mapsto\Delta_n$ with the functor~$\nu$ we obtain a functor from the simplex category to the category of $\omega$-categories. The $n$th oriental~$\Or_n$ may be defined by the formula
$$\Or_n=\nu\Delta_n.$$
The nerve of an $\omega$-category is then defined as follows.

\begin{definition} \label{D2.4}
The \emph{nerve} of an $\omega$-category~$C$ is the simplicial set $NC$ given by
$$N_n C=\Hom(\Or_n,C)=\Hom(\nu\Delta_n,C).$$
\end{definition}

\section{Nerves as stratified simplicial sets} \label{S3}

A complicial set is defined to be a simplicial set with a class of thin elements subject to certain conditions; in particular they must be stratified simplicial sets in the sense of Definition~\ref{D3.2} below. In this section we specify the thin elements in the nerves of $\omega$-categories and we then show that the nerves are stratified simplicial sets.

\begin{definition} \label{D3.1}
Let $C$ be an $\omega$-category. Then the \emph{thin} elements in $N C$ are the elements~$x$ in $N_n C$ with $n>0$ such that
$$d_{n-1}^- x u=d_{n-1}^+ x u=x u$$
for all $u\in\nu\Delta_n$
\end{definition}

Note that the elements of $\nu\Delta_n$ are of dimension at most~$n$, in the sense that $d_n^-u=d_n^+ u=u$ for all $u\in\nu\Delta_n$. A thin element is therefore a morphism that lowers dimension. 

We recall the definition of stratified simplicial sets from \cite{B9},~5.1.

\begin{definition} \label{D3.2}
A \emph{stratified simplicial set} is a simplicial set together with a prescribed class of thin elements of positive dimension which includes the degenerate elements. A \emph{morphism of stratified simplicial sets} is a morphism of simplicial sets taking thin elements to thin elements.
\end{definition}

We obtain the following result.

\begin{proposition} \label{P3.3}
The nerve functor is a functor from $\omega$-categories to stratified simplicial sets.
\end{proposition}

\begin{proof}
Let $C$ be an $\omega$-category. To show that $NC$ is a stratified simplicial set we must show that the degenerate elements in $NC$ are thin. To do this, let $x$ be a degenerate $n$-dimensional element, say 
$$x=\e_i y=y(\nu\e_i^\v)$$
with $y\colon\nu\Delta_{n-1}\to C$. We must show that $d_{n-1}^\alpha x u=x u$ for each sign~$\alpha$ and for all $u\in\nu\Delta_n$. But we have $d_{n-1}^\alpha v=v$ for all $v\in\nu\Delta_{n-1}$, hence
$$d_{n-1}^\alpha x u
=d_{n-1}^\alpha y(\nu\e_i^\v)u
=y d_{n-1}^\alpha(\nu\e_i^\v)u
=y(\nu\e_i^\v)u=x u$$
as required.

Now let $f\colon C\to C'$ be a morphism of $\omega$-categories. To show that $Nf$ is a morphism of stratified simplicial sets we must show that $(Nf)x$ is thin whenever $x$~is a thin element in $N_n C$. But for $u\in\nu\Delta_n$ and for each sign~$\alpha$ we have
$$d_{n-1}^\alpha(Nf)xu=d_{n-1}^\alpha f xu=fd_{n-1}^\alpha xu=fxu=(Nf)xu;$$
therefore $(Nf)x$ is thin as required.
 
This completes the proof.
\end{proof}

\section{Complicial sets} \label{S4}

A complicial set is a stratified simplicial set satisfying certain conditions on horn-fillers. In this section we give the definition, following \cite{B9}, 6--7.

Recall that, for $n>0$ and for $0\leq k\leq n$, an \emph{$(n-1)$-dimensional $k$-horn} in an arbitrary simplicial set is a sequence of $(n-1)$-dimensional elements 
$$z=(z_0,\ldots,z_{k-1},z_{k+1},\ldots,z_n)$$
such that
$$\d_i z_j=\d_{j-1}z_i\quad (i<j);$$
thus the elements are like the codimension one faces of an $n$-dimensional element~$x$, except that the face $\d_k x$ is omitted. A \emph{filler} for an $(n-1)$-dimensional $k$-horn $z=(z_0,\ldots,z_n)$ is an $n$-dimensional element~$x$ such that $\d_i x=z_i$ for $i\neq k$.

In a stratified simplicial set there are distinguished classes of elements and horns which are called complicial or admissible. Consider a geometrical $n$-simplex with vertices $0,1,\ldots,n$. It has faces spanned by the non-empty subsets of $\{0,\ldots,n\}$; the face spanned by the complement of a proper subset $\{j(1),\ldots,j(r)\}$ with $j(1)<\ldots<j(r)$ corresponds to the iterated face operator
$\d_{j(1)}\ldots\d_{j(r)}$. Given integers $n$~and~$k$ with $0<k<n$ we will say that an $n$-dimensional element~$x$ is $k$-complicial (the term $k$-admissible is also used) if the iterated face $\d_{j(1)}\ldots\d_{j(r)}x$ is thin whenever the complement of $\{j(1),\ldots,j(r)\}$ includes the integers $k-1,k,k+1$. We will also say that an $(n-1)$-dimensional horn $z=(z_0,\ldots,z_n)$ is $k$-complicial if the elements~$z_i$ are configured like the codimension one faces of a $k$-complicial $n$-dimen\-sion\-al element~$x$ with $\d_k x$ omitted.

Formally, it is most convenient to express these definitions recursively.

\begin{definition} \label{D4.1}
Let $n$~and~$k$ be integers with $0<k<n$. 

\textup{(1)} An \emph{$(n-1)$-dimensional $k$-complicial horn} in a stratified simplicial set is an $(n-1)$-dimensional $k$-horn
$$z=(z_0,\ldots,z_{k-1},z_{k+1},\ldots,z_n)$$
such that $z_i$~is a $(k-1)$-complicial element for $0\leq i<k-1$ and is a $k$-complicial element for $k+1<i\leq n$.

\textup{(2)} An \emph{$n$-dimensional $k$-complicial element} in a stratified simplicial set is a thin filler of an $(n-1)$-dimensional $k$-complicial horn.
\end{definition}

Thus a $1$-dimensional $1$-horn is always $1$-complicial, a $2$-dimensional element is $1$-complicial if and only if it is thin, etc.

A complicial set is then defined as follows.

\begin{definition} \label{D4.2}
A \emph{complicial set} is a stratified simplicial set satisfying the following conditions.

\textup{(1)} Every thin $1$-dimensional element is degenerate.

\textup{(2)} For $0<k<n$ every $(n-1)$-dimensional $k$-complicial horn has a unique thin filler.

\textup{(3)} For $0<k<n$, if $x$~is an $n$-dimensional $k$-complicial element and if $\d_{k-1}x$ and $\d_{k+1}x$ are thin, then $\d_k x$ is thin.
\end{definition}

We aim to show that the nerves of $\omega$-categories are complicial sets. Most of the work is in verifying the last two conditions; we can dispose of the first condition immediately.

\begin{proposition} \label{P4.3}
Let $x$ be a thin $1$-dimensional element in the nerve of an $\omega$-category~$C$. Then $x$~is degenerate.
\end{proposition}

\begin{proof}
Recall that $x$~is a morphism of $\omega$-categories from $\nu\Delta_1$ to~$C$. It is straightforward to check that $\nu\Delta_1$ has exactly three members, which can be listed as
\begin{align*}
&u=(\,[0],[1]\mid [0,1],[0,1]\mid 0,0\mid\ldots\,),\\
&d_0^- u=(\,[0],[0]\mid 0,0\mid\ldots\,),\\
&d_0^+ u=(\,[1],[1]\mid 0,0\mid\ldots\,).
\end{align*}
We have $xd_0^- u=d_0^- xu=xu$ and $xd_0^+ u=d_0^+ xu=xu$ because $x$~is thin; therefore $x$~is constant. It follows that $x=x(\nu\d_0^\v)(\nu\e_0^\v)=\e_0\d_0 x$; therefore $x$~is degenerate.
\end{proof} 

\section{Horns in nerves} \label{S5}

Recall that the $n$-dimensional elements in the nerves of $\omega$-categories are represented by the free augmented directed complex~$\Delta_n$ via the functor~$\nu$. In this section we show that $(n-1)$-dimensional $k$-horns are similarly represented by a subcomplex~$\Lambda_n^k$ of~$\Delta_n$.

The subcomplex~$\Lambda_n^k$ is as one would expect.

\begin{notation} \label{N5.1}
If $n>0$ and if $0\leq k\leq n$ then $\Lambda_n^k$~is the free augmented directed subcomplex of~$\Delta_n$ given by
$$\Lambda_n^k
=\d_0^\v\Delta_{n-1}+\ldots+\d_{k-1}^\v\Delta_{n-1}
+\d_{k+1}^\v\Delta_{n-1}+\ldots+\d_n^\v\Delta_{n-1}.$$
The prescribed basis for~$\Lambda_n^k$ is the prescribed basis for~$\Delta_n$ with the elements
$$[0,1,\ldots,k-1,k+1,\ldots,n-1,n],\quad [0,1,\ldots,n-1,n]$$
omitted.
\end{notation}

For $n\geq 2$ and for $0\leq i<j\leq n$ we have
$$\d_i^\v\Delta_{n-1}\cap\d_j^\v\Delta_{n-1}
=\d_j^\v\d_i^\v\Delta_{n-2}
=\d_i^\v\d_{j-1}^\v\Delta_{n-2}.$$
In the category of abelian groups this gives a description of~$\Lambda_n^k$ as a colimit. To show that $\Lambda_n^k$~represents horns in nerves, we must show that $\nu\Lambda_n^k$ is a similar colimit in the category of $\omega$-categories. We will do this by using results from~\cite{B2}, which we now summarise.

Let $K$ be a free augmented directed complex. Given a chain~$c$ in~$K$, we write $\d^+ c$ and $\d^- c$ for the positive and negative parts of the boundary $\d c$; in other words, $\d^+ c$ and $\d^- c$ are the sums of basis elements without common terms such that
$$\d c=\d^+ c-\d^- c.$$
Given basis elements $a,a'$ of the same dimension~$q$, we will write $a<a'$ if there is a basis element~$b$ of dimension $q+r$ with $r>0$ such that $a$~is a term in $(\d^-)^r b$ and such that $a'$~is a term in $(\d^+)^r b$.

\begin{definition} \label{D5.2}
A free augmented directed complex is \emph{unital} if $\e(\d^-)^p a=\e(\d^+)^p a=1$ for every basis element~$a$, where $p$~is the dimension of~$a$. 

A free augmented directed complex is \emph{loop-free} if the transitive closure of the relation~$<$ is a strict partial order on its basis elements.
\end{definition}

\begin{definition} \label{D5.3}
Let $a$ be a $p$-dimensional basis element in a unital loop-free directed complex~$K$. The the associated \emph{atom} is the member~$\langle a\rangle$ of $\nu K$ given by
$$\langle a\rangle
=(\,(\d^-)^p a,(\d^+)^p a\mid\ldots\mid \d^- a,\d^+a\mid a,a\mid 0,0\mid\ldots\,).$$
\end{definition}

Note that $\d\d^-=\d\d^+$ because $\d\d=0$, hence $\d^\alpha\d^-=\d^\alpha\d^+$ for each sign~$\alpha$; it follows that in Definition~\ref{D5.3} $\langle a\rangle$~really is a member of $\nu K$.

The main result that we will use is the following (\cite{B2}, Theorem~6.1). 

\begin{theorem} \label{T5.4}
If $K$~is a unital loop-free free augmented directed complex then the $\omega$-category $\nu K$ has a presentation as follows. The generators are the atoms. For a basis element~$a$ of arbitrary dimension~$p$ there are relations
$$d_p^-\langle a\rangle=d_p^+\langle a\rangle=\langle a\rangle.$$ 
For a basis element~$a$ of positive dimension~$p$ there are relations
$$d_{p-1}^-\langle a\rangle =w^-(a),\quad
d_{p-1}^+\langle a\rangle =w^+(a),$$
where $w^-(a)$ and $w^+(a)$ are arbitrarily chosen expressions for $d_{p-1}^-\langle a\rangle$ and $d_{p-1}^+\langle a\rangle$ as iterated composites of atoms of dimension less than~$p$.
\end{theorem}

We will now show that this theorem applies to simplexes, essentially by computing the atoms.

\begin{definition} \label{D5.5}
Let
$$a=[i_0,\ldots,i_p]$$ 
be a $p$-dimensional basis element in a simplex~$\Delta_n$. Then a \emph{block} in~$a$ is a non-empty sequence of consecutive terms
$$\mathbf{u}=i_r,i_{r+1},\ldots,i_s.$$
A block is \emph{even} if it has an even number of terms, it is \emph{odd} if it has an odd number of terms, it is \emph{initial} if it includes~$i_0$, and it is \emph{final} if it includes~$i_p$.
\end{definition}

\begin{proposition} \label{P5.6}
Let $a$ be a $p$-dimensional element in a simplex~$\Delta_n$ and let $q$ be an integer with $0\leq q\leq p$. Then $(\d^-)^{p-q}a$ is the sum of the $q$-dimensional basis elements with expressions
$$[\mathbf{u}_0,\ldots,\mathbf{u}_k]\quad (k\geq 0)$$
such that $\mathbf{u_0}$~is an odd initial block in~$a$ and such that $\mathbf{u}_1,\ldots,\mathbf{u}_k$ are even or final blocks in~$a$. In the same way $(\d^+)^{p-q}a$ is the sum of the $q$-dimensional basis elements with expressions
$$[\mathbf{u}_0,\ldots,\mathbf{u}_k]\quad (k\geq 0)$$
such that $\mathbf{u}_0,\ldots,\mathbf{u}_k$ are even or final blocks in~$a$.
\end{proposition}

\begin{proof}
This is a straightforward downward induction on~$q$.
\end{proof}

\begin{proposition} \label{P5.7}
The free augmented directed complexes $\Delta_n$ and $\Lambda_n^k$ are unital and loop-free.
\end{proposition}

\begin{proof}
We make the following deductions from Proposition~\ref{P5.6}.

If $a=[i_0,\ldots,i_p]$ is a $p$-dimensional basis element, then 
$$\e(\d^-)^p a=\e[i_0]=1,\quad \e(\d^+)^p a=\e[i_p]=1;$$
therefore the complexes are unital.

If $a$~and~$a'$ are basis elements such that $a<a'$, say 
$$a=[i_0,\ldots,i_p],\quad a'=[i_0',\ldots,i_p']$$
and if $i_0=i_0'$, \dots, $i_{q-1}=i'_{q-1}$, $i_q\neq i'_q$, then $(-1)^q i_q<(-1)^q i'_q$. The transitive closure of~$<$ is therefore contained in a total ordering of the basis elements and it follows that the complexes are loop-free.
\end{proof}

For $n\geq 2$ it now follows from Theorem~\ref{T5.4} that $\nu\Lambda_n^k$ is the obvious colimit of copies of $\nu\Delta_{n-1}$ and $\nu\Delta_{n-2}$. If $n=1$ then $\nu\Lambda_n^k$~is obviously just a copy of $\nu\Delta_{n-1}$. We deduce the following result.

\begin{theorem} \label{T5.8}
Let $n$~and~$k$ be integers with $n>0$ and $0\leq k\leq n$, and let $C$ be an $\omega$-category. Then the function
$$\Hom(\nu\Lambda_n^k,C)
\to\Hom(\nu\Delta_{n-1},C)\times\ldots\times\Hom(\nu\Delta_{n-1},C)$$
induced by $(\,\nu\d_i^\v:i\neq k\,)$ is a bijection whose image consists of the $(n-1)$-dimensional $k$-horns in the nerve of~$C$. The function
$$\Hom(\nu\Delta_n,C)\to\Hom(\nu\Lambda_n^k,C)$$
induced by the inclusion $\nu\Lambda_n^k\to\nu\Delta_n$ sends $n$-dimensional elements to the $k$-horns which they fill.
\end{theorem}

\section{Horns represented by pairs of faces} \label{S6}

Let $n$~and~$k$ be integers with $0<k<n$. We have constructed a subcomplex~$\Lambda_n^k$ in~$\Delta_n$ which represents $(n-1)$-dimensional $k$-horns (Theorem~\ref{T5.8}). We will now construct a subcomplex~$V_n^k$ of~$\Lambda_n^k$ and a retraction
$$\Pi_n^k\colon\Delta_n\to V_n^k.$$
Eventually we will show that $V_n^k$~represents $n$-dimensional $k$-complicial elements, and that $\Pi_n^k$ represents the inclusion of these elements in the set of all $n$-dimensional elements. We will also show that $V_n^k$~represents $(n-1)$-dimensional $k$-complicial horns and that the restriction
$$\pi_n^k=\Pi_n^k|\Lambda_n^k\colon \Lambda_n^k\to V_n^k$$
represents the inclusion of these horns in the set of all $(n-1)$-dimen\-sion\-al $k$-horns.

We will now define the subcomplex~$V_n^k$ and the retraction~$\Pi_n^k$.

\begin{notation} \label{N6.1}
If $n>0$ and $0<k<n$ then $V_n^k$~is the free augmented directed subcomplex of~$\Delta_n$ given by
$$V_n^k=\d_{k-1}^\v\Delta_{n-1}+\d_{k+1}^\v\Delta_{n-1}.$$
The prescribed basis elements for~$V_n^k$ are the prescribed basis elements $[i_0,\ldots,i_q]$ for~$\Delta_n$ not including both $k-1$ and $k+1$.
\end{notation}

The results for $\Delta_n$~and~$\Lambda_n^k$ given in Proposition~\ref{P5.7} obviously apply to~$\Lambda_n^k$ as well.

\begin{proposition} \label{P6.2}
The free augmented directed complex~$V_n^k$ is unital and loop-free.
\end{proposition}

\begin{notation} \label{N6.3}
Let $n$~and~$k$ be integers with $0<k<n$. Then 
$$\Pi_n^k\colon\Delta_n\to V_n^k$$
is the morphism given on basis elements $a=[i_0,\ldots,i_q]$ as follows. If $a$~does not include $k-1$ or does not include $k+1$ then $\Pi_n^k a=a$. If $a$~includes $k-1$, $k$ and $k+1$ then $\Pi_n^k a=0$. If $a$ includes $k-1$ and $k+1$ but not~$k$, say
$$a=[\mathbf{u},k-1,k+1,\mathbf{v}],$$
then
$$\Pi_n^k a=[\mathbf{u},k,k+1,\mathbf{v}]+[\mathbf{u},k-1,k,\mathbf{v}].$$
We also write~$\pi_n^k$ for the restriction
$$\pi_n^k=\Pi_n^k|\Lambda_n^k\colon\Lambda_n^k\to V_n^k.$$
\end{notation}

It is straightforward to show $\Pi_n^k$~and~$\pi_n^k$ are augmentation-preserving chain maps.  Since they take basis elements to sums of basis elements, they are morphisms of free augmented directed complexes. It is also straightforward to check that they have the following properties.

\begin{proposition} \label{P6.4}
The morphisms of free augmented directed complexes $\Pi_n^k$~and~$\pi_n^k$ are retractions such that
\begin{align*}
&\pi_n^k\d_i^\v=\d_i^\v\Pi_{n-1}^{k-1}\quad (0\leq i<k-1),\\
&\pi_n^k\d_{k-1}^\v=\d_{k-1}^\v,\\
&\pi_n^k\d_{k+1}^\v=\d_{k+1}^\v,\\
&\pi_n^k\d_i^\v=\d_i^\v\Pi_{n-1}^k\quad (k+1<i\leq n).
\end{align*}
\end{proposition}

We immediately obtain an analogue for Definition \ref{D4.1}(1).

\begin{proposition} \label{P6.5}
Let $C$ be an $\omega$-category, let $z$ be a morphism in $\Hom(\nu\Lambda_n^k,C)$, and let
$$z_i=z(\nu\d_i^\v)\quad (0\leq i\leq n,\ i\neq k).$$
Then $z$~is of the form $y(\nu\pi_n^k)$ if and only if $z_i$ is of the form $y_i(\nu\Pi_{n-1}^{k-1})$ for $0\leq i<k-1$ and of the form $y_i(\nu\Pi_{n-1}^k)$ for $k+1<i\leq n$.
\end{proposition}

This shows that composites with morphisms of the forms $\nu\pi_n^k$ and $\nu\Pi_{n-1}^k$ are somewhat analogous to $(n-1)$-dimensional complicial horns and $(n-1)$-dimensional complicial elements (see Definition~\ref{D4.1}(1)). To complete the analogy we must relate composites with $\nu\Pi_n^k$ to thin fillers of composites with $\nu\pi_n^k$. We will do this in the next section.

\section{Fillers for horns represented by pairs of faces} \label{S7}

In this section let $n$~and~$k$ be fixed integers with $0<k<n$. We have seen in Theorem~\ref{T5.8} that the process of filling $(n-1)$-dimensional $k$-horns in the nerves of $\omega$-categories is represented by the inclusion of~$\Lambda_n^k$ in~$\Delta_n$. In Section~\ref{S6} we have constructed a retraction~$\pi_n^k$ from~$\Lambda_n^k$ to~$V_n^k$ which is intended to represent the inclusion of the set of $(n-1)$-dimensional $k$-complicial horns in the set of arbitrary $(n-1)$-dimen\-sion\-al $k$-horns. We will now construct an extension~$W$ of~$V_n^k$ and a morphism $\Pi\colon\Delta_n\to W$ extending~$\pi_n^k$ such that the square
$$\xymatrix{
\nu\Lambda_n^k \ar_{\nu\pi_n^k}[d] \ar^{\subset}[r] & \nu\Delta_n \ar^{\nu\Pi}[d] \\
\nu V_n^k \ar^{\subset}[r] & \nu W
}$$
is a push-out square. The fillers for horns represented by~$V_n^k$ will then be represented by~$W$. We will also show that any such horn has a unique thin filler.

We will use the following notation.

\begin{notation} \label{N7.1}
In~$\Delta_n$ let $s,t_0,\ldots,t_n$ be the basis elements of dimensions $n$~and $n-1$ given by
\begin{align*}
&s=[0,1,\ldots,n],\\
&t_i=[0,\ldots,i-1,i+1,\ldots,n]\quad (0\leq i\leq n).
\end{align*}

Let $\kappa$ be the sign given by
$$\kappa=(-)^k.$$
\end{notation}

Note that $\Delta_n$~is obtained from~$\Lambda_n^k$ by adjoining the basis elements $t_k$~and~$s$. Note also that
$$\d s=\sum_{i=0}^n (-1)^i t_i$$ 
with
\begin{align*}
&\pi_n^k t_{k-1}=t_{k-1},\\
&\pi_n^k t_{k+1}=t_{k+1},\\
&\pi_n^k t_i=0\quad (i\neq k-1,k,k+1).
\end{align*}
We can therefore construct a free augmented directed complex~$W$ and a morphism of free augmented directed complexes $\Pi\colon\Delta_n\to W$ in the following way.

\begin{notation} \label{N7.2}
Let $W$ be the free augmented directed complex obtained from~$V_n^k$ by adjoining an $(n-1)$-dimensional basis element~$t_k'$ and an $n$-dimensional basis element~$s'$ with
\begin{align*}
&\d t_k'=\pi_n^k\d t_k,\\
&\d s'=(-1)^k(t_k'-t_{k-1}-t_{k+1}).
\end{align*}
Let $\Pi\colon\Delta_n\to W$ be the morphism given by
$$\Pi|\Lambda_n^k=\pi_n^k,\quad \Pi t_k=t_k',\quad \Pi s=s'.$$
\end{notation}

We will obtain properties of $\nu W$ by using Theorem~\ref{T5.4}. In order to do this, we must show that $W$~is unital and loop-free. We therefore make the following calculations.

\begin{proposition} \label{P7.3}
Let $a$ be a basis element in~$W$ which is a term in $\Pi(\d^\alpha)^{r+1}s'$ or $\Pi(\d^\alpha)^r t_k'$ for some sign~$\alpha$. If $a$~is a basis element in~$V_n^k$ omitting $k-1$ then $a$~is a term in $(\d^\alpha)^r t_{k-1}$; if $a$~is a basis element in~$V_n^k$ omitting $k+1$ then $a$~is a term in $(\d^\alpha)^r t_{k+1}$.
\end{proposition}

\begin{proof}
Let $b$ be a basis element in~$\Delta_n$ which is a term in $(\d^\alpha)^{r+1}s$ or $(\d^\alpha)^r t_k$; then $b$~consists of blocks of consecutive entries in $s$ or~$t_k$ satisfying the conditions of Proposition~\ref{P5.6}. It is straightforward to check that any term of $\Pi b$ in~$V_n^k$ omitting $k-1$ consists of blocks of consecutive entries in~$t_{k-1}$ satisfying the same conditions and is therefore a term in $(\d^\alpha)^r t_{k-1}$. A similar argument applies to terms omitting $k+1$.
\end{proof}

\begin{proposition} \label{P7.4}
Let $\alpha$ be any sign; then
\begin{align*}
&(\d^\alpha)^r s'=\Pi(\d^\alpha)^r s\quad (0\leq r\leq n),\\
&(\d^\alpha)^r t_k'=\Pi(\d^\alpha)^r t_k\quad (0\leq r\leq n-1).
\end{align*}
\end{proposition}

\begin{proof}
We first prove the result for $(\d^\alpha)^r s$ by induction on~$r$. The result holds for $r=0$ and for $r=1$ because
\begin{align*}
&s'=\Pi s,\\
&\d^\kappa s'= t_k'=\Pi\d^\kappa s,\\
&\d^{-\kappa}s'=t_{k-1}+t_{k+1}=\Pi\d^{-\kappa}s.
\end{align*}
In cases with $2\leq r\leq n$ assume that $(\d^-)^{r-1}s'=\Pi(\d^-)^{r-1}s$. Then
$$\d(\d^-)^{r-1}s'
=\d\Pi(\d^-)^{r-1}s
=\Pi\d(\d^-)^{r-1}s
=\Pi(\d^+)^r s-\Pi(\d^-)^r s.$$
By construction, $(\d^+)^r s'$ and $(\d^-)^r s'$ are sums of basis elements without common terms such that
$$\d(\d^-)^{r-1}s'=(\d^+)^r s'-(\d^-)^r s'.$$
It follows from Proposition~\ref{P7.3} that $\Pi(\d^+)^r s$ and $\Pi(\d^-)^r s$ are also sums of basis elements without common terms, and it then follows that
$$(\d^\alpha)^r s'=\Pi(\d^\alpha)^r s$$
for each sign~$\alpha$.

A similar argument applies to $(\d^\alpha)^r t_k'$, starting with $t_k'=\Pi t_k$.
\end{proof}

\begin{proposition} \label{P7.5}
Let $a$~and~$b$ be basis elements in~$W$ such that $a$~is a term in $(\d^\alpha)^r s'$ or $(\d^\alpha)^r t_k'$ for some~$\alpha$ and such that $a$~is also a term in $(\d^{-\alpha})^q b$. Then $r=0$ or $q=0$.
\end{proposition}

\begin{proof}
By construction, if $r>0$ then $(\d^-)^r s'$ and $(\d^+)^r s'$ have no common terms. For $r>1$ we have
$$(\d^-)^r s'=(\d^-)^{r-1}t_k',\quad (\d^+)^r s'=(\d^+)^{r-1}t_k',$$
because $\d^\kappa s'=t_k'$. This gives the result for $b=s'$ and for $b=t_k'$.

Now let $b$ be a basis element in~$V_n^k$ omitting $k-1$. Since $a$~is a term in $(\d^{-\alpha})^q b$, it follows that  $a$~is also a basis element in~$V_n^k$ omitting $k-1$. By Proposition~\ref{P7.3}, $a$~is a term in $(\d^\alpha)^{r-1}t_{k-1}$ or $(\d^\alpha)^r t_{k-1}$. It now follows from Proposition~\ref{P5.6} that $a$~is a term in $(\d^\alpha)^q b$, because terms of~$a$ which are consecutive in~$t_{k-1}$ are also consecutive in~$b$. This makes~$a$ a term in $(\d^\alpha)^q b$ and in $(\d^{-\alpha})^q b$; therefore $q=0$.

A similar argument applies when $b$~is a basis element in~$V_n^k$ omitting $k+1$.
\end{proof}

\begin{theorem} \label{T7.6}
The free augmented directed complex~$W$ is unital and loop-free.
\end{theorem}

\begin{proof}
By Proposition~\ref{P6.2} the subcomplex~$V_n^k$ is unital and loop-free; it therefore suffices to consider the effect of the additional basis elements $s'$~and~$t_k'$.

Since $\Pi$~is augmentation-preserving and $\Delta_n$~is unital, it follows from Proposition~\ref{P7.4} that  
$$\e(\d^\alpha)^n s'=\e\Pi(\d^\alpha)^n s=\e(\d^\alpha)^n s=1$$
and it similarly follows that $\e(\d^\alpha)^{n-1}t_k'=1$; therefore $W$~is unital.

Next we consider the relations $a<a'$ between basis elements of~$W$. The transitive closure of the relations coming from~$V_n^k$ is a partial ordering because $V_n^k$~is loop-free. If $a<a'$ is a relation coming from the extra basis elements $s'$~and~$t_k'$, then it follows from Proposition~\ref{P7.5} that there are no relations of the form $c<a$ or $a'<c'$. The transitive closure therefore remains a partial ordering even after the extra relations are included, and it follows that $W$~is loop-free.

This completes the proof.
\end{proof}

It follows from Theorems \ref{T7.6} and~\ref{T5.4} that $\nu W$ has a presentation in terms of its atoms. For these atoms we have the following results. 

\begin{proposition} \label{P7.7}
The atoms $\langle s'\rangle$~and~$\langle t_k'\rangle$ in $\nu W$ are such that
$$\langle s'\rangle=(\nu\Pi)\langle s\rangle,\quad
\langle t_k'\rangle=(\nu\Pi)\langle t_k\rangle$$
and
$$d_{n-1}^\kappa\langle s'\rangle=\langle t_k'\rangle,\quad
d_{n-1}^{-\kappa}\langle s'\rangle\in\nu V_n^k.$$
\end{proposition}

\begin{proof}
The equalities $\langle s'\rangle=(\nu\Pi)\langle s\rangle$ and $\langle t_k'\rangle=(\nu\Pi)\langle t_k\rangle$ are restatements of Proposition~\ref{P7.4}. The equality $d_{n-1}^\kappa\langle s'\rangle=\langle t_k'\rangle$ holds because $\d^\kappa s'=t_k'$. We have $d_{n-1}^{-\kappa}\langle s'\rangle\in\nu V_n^k$ because
$$\d^{-\kappa}s'=t_{k-1}+t_{k+1}\in V_n^k.$$
\end{proof}

It follows that $\nu W$ is a push-out in the required way.

\begin{theorem} \label{T7.8}
The square of $\omega$-categories
$$\xymatrix{
\nu\Lambda_n^k \ar_{\nu\pi_n^k}[d] \ar^{\subset}[r] & \nu\Delta_n \ar^{\nu\Pi}[d] \\
\nu V_n^k \ar^{\subset}[r] & \nu W
}$$
is a push-out square. The $\omega$-category $\nu W$ is generated as an extension of $\nu V_n^k$ by the atom~$\langle s'\rangle$.
\end{theorem}

\begin{proof}
It follows from Theorem~\ref{T5.4}, Theorem~\ref{T7.6} and Proposition~\ref{P7.7} that $\nu W$ is generated as an extension of $\nu V_n^k$ by $\langle s'\rangle$~and~$\langle t_k'\rangle$ subject to relations which are the images under $\nu\Pi$ of relations in $\nu\Delta_n$. This shows that the square is a push-out square. Since $\langle t_k'\rangle=d_{n-1}^\kappa\langle s'\rangle$, the single atom~$\langle s'\rangle$ suffices to generate $\nu W$ as an extension of $\nu V_n^k$.
\end{proof}

We can now prove the main theorem of this section.

\begin{theorem} \label{T7.9}
Let $C$ be an $\omega$-category and let $y$ be a morphism in $\Hom(\nu V_n^k,C)$. Then the composite 
$$y(\nu\Pi_n^k)\in\Hom(\nu\Delta_n,C)$$
is the unique thin filler for the $(n-1)$-dimensional $k$-horn represented by
$$y(\nu\pi_n^k)\in\Hom(\nu\Lambda_n^k,C).$$
\end{theorem}

\begin{proof}
The morphism $y(\nu\Pi_n^k)$ is certainly a filler for $y(\nu\pi_n^k)$ because
$\Pi_n^k|\Lambda_n^k=\pi_n^k$. It is a thin filler because $d_{n-1}^\alpha v=v$ for all $v\in\nu V_n^k$.

It now suffices to show that there is at most one thin filler. To do this, let $x$ be an arbitrary thin filler. By Theorem~\ref{T7.8} there is a morphism
$$w\in\Hom(\nu W,C)$$
such that
$$w|\nu V_n^k=y,\quad w(\nu\Pi)=x.$$
By Proposition~\ref{P7.7}, $w\langle s'\rangle=x\langle s\rangle$. Since $x$~is thin, 
$w\langle s'\rangle=d_{n-1}^{-\kappa} w\langle s'\rangle$. By another application of Proposition~\ref{P7.7}, 
$$d_{n-1}^{-\kappa}w\langle s'\rangle=w d_{n-1}^{-\kappa}\langle s'\rangle=y d_{n-1}^{-\kappa}\langle s'\rangle;$$
therefore
$$w\langle s'\rangle
=y d_{n-1}^{-\kappa}\langle s'\rangle.$$
Since $\nu W$ is generated by~$\langle s'\rangle$ as an extension of $\nu V_n^k$ and since $x$~is $w(\nu\Pi)$ it follows that $x$~is uniquely determined by~$y$. This completes the proof.
\end{proof}

\section{Nerves as complicial sets} \label{S8}

We have shown in Proposition~\ref{P3.3} that the nerves of $\omega$-categories are naturally stratified simplicial sets. We will now show that they are complicial sets, with complicial elements represented by the complexes~$V_n^k$.

\begin{theorem} \label{T8.1}
Let $C$ be an $\omega$-category. For integers $n$~and~$k$ with $0<k<n$, an $(n-1)$-dimensional $k$-horn in the nerve of~$C$ is $k$-complicial if and only if it is represented by a composite $y(\nu\pi_n^k)$ with $y\in\Hom(\nu V_n^k,C)$, and an $n$-dimensional element in the nerve of~$C$ is $k$-complicial if and only if it is represented by a composite $y(\nu\Pi_n^k)$ with $y\in\Hom(\nu V_n^k,C)$.
\end{theorem}

\begin{proof}
The proof is by induction on~$n$.

Suppose that $(n-1)$-dimensional elements are $(k-1)$-complicial if and only if they factor through $\nu\Pi_{n-1}^{k-1}$ and that they are $k$-complicial if and only if they factor through $\nu\Pi_{n-1}^k$. Then it follows from Proposition~\ref{P6.4} that an $(n-1)$-dimensional $k$-horn 
$$(z_0,\ldots,z_{k-1},z_{k+1},\ldots,z_n)$$
factors through $\nu\pi_n^k$ if and only if $z_i$~is $(k-1)$-complicial for $0\leq i<k-1$ and $z_i$~is $k$-complicial for $k+1<i\leq n$. This means that an $(n-1)$-dimensional $k$-horn factors through $\nu\pi_n^k$ if and only if it is $k$-complicial. It now follows from Theorem~\ref{T7.9} that an $n$-dimensional element is a thin filler of an $(n-1)$-dimensional $k$-complicial horn if and only if it factors through $\nu\Pi_n^k$; therefore an $n$-dimensional element is $k$-complicial if and only if it factors through $\nu\Pi_n^k$. This completes the proof.
\end{proof}

\begin{theorem} \label{T8.2}
If $C$~is an $\omega$-category then the nerve of~$C$ is a complicial set.
\end{theorem}

\begin{proof}
By Proposition~\ref{P3.3} the nerve is a stratified simplicial set. By Proposition~\ref{P4.3} every thin $1$-dimensional element in the nerve is degenerate. By Theorems \ref{T7.9} and~\ref{T8.1} complicial horns in the nerve have unique thin fillers. It remains to verify the condition of Definition \ref{D4.2}(3).

Let $x$ be an $n$-dimensional $k$-complicial element in the nerve such that $\d_{k-1}x$ and $\d_{k+1}x$ are thin; we must show that $\d_k x$ is thin. Equivalently, for each sign~$\alpha$, we must show that $d_{n-2}^\alpha x(\nu\d_k^\v)=x(\nu\d_k^\v)$. We will do this by showing that $d_{n-2}^\alpha x=x$. Since $x$~factors through the retraction $\nu\Pi_n^k\colon\nu\Delta_n\to\nu V_n^k$ it suffices to show that $d_{n-2}^\alpha xv =xv$ for all $v\in\nu V_n^k$. Since $\nu V_n^k$ is generated by its atoms (see Theorem~\ref{T5.4} and Proposition~\ref{P6.2}) it suffices by Proposition~\ref{P2.1} to show that $d_{n-2}^\alpha x\langle a\rangle=x\langle a\rangle$ for every basis element~$a$ in~$V_n^k$. For basis elements~$a$ of dimension less than $n-1$ this is immediate. For a basis element~$a$ of dimension $n-1$ we have 
$$x\langle a\rangle=x\langle t_{k\pm 1}\rangle
\in\im x(\nu\d_{k\pm 1}^\v)=\im\d_{k\pm 1}x,$$
hence $d_{n-2}^\alpha x\langle a\rangle=x\langle a\rangle$ because $\d_{k\pm 1}x$ is thin. This completes the proof.
\end{proof}

\section{Complicial sets and complicial identities} \label{S9}

We have shown in Theorem~\ref{T8.2} that the nerves of $\omega$-categories are complicial sets. In this section we will show that complicial sets are sets with complicial identities in the sense of the following definition (\cite{B4}, Definition~6.1).

\begin{definition} \label{D9.1}
A \emph{set with complicial identities} is a simplicial
set~$X$, together with \emph{wedges}
$$x\w_i y\in X_{m+1},$$
defined when $x,y\in X_m$ and $\d_i x=\d_{i+1}y$, such that the
following axioms hold.

\textup{(1)} If $x\w_i y$ is defined with $x,y\in X_m$, then
\begin{align*}
 &\d_j(x\w_i y)=\d_j x\w_{i-1}\d_j y\quad (0\leq j<i),\\
 &\d_i(x\w_i y)=y,\\
 &\d_{i+2}(x\w_i y)=x,\\
 &\d_j(x\w_i y)=\d_{j-1}x\w_i\d_{j-1}y\quad (i+2<j\leq m+1).
\end{align*}

\textup{(2)} If $x\in X_m$ and $0\leq i<m$ then
$$\e_i x=\e_i\d_{i+1}x\w_i x,\quad \e_{i+1}x=x\w_i\e_i\d_i x.$$

\textup{(3)} If $A$~is of the form $b\w_i(y\w_i z)$ then
$$A=(\d_{i+2}b\w_i y)\w_{i+1}\d_{i+1}A.$$

\textup{(4)} If $A$~is of the form $(x\w_i y)\w_{i+1}c$ then
$$A=\d_{i+2}A\w_i(y\w_i\d_i c).$$

\textup{(5)} The equality
$$[x\w_i\d_{i+1}(y\w_i z)]\w_i(y\w_i z)
 =(x\w_i y)\w_{i+1}[\d_{i+1}(x\w_i y)\w_i z]$$
holds whenever either side is defined.

\textup{(6)} If $A$~is of the form
$\d_{i+2}[(x\w_{i+1}y)\w_{i+1}(y\w_i z)]$, then the
equality
$$A\w_i(w\w_{i+1}\d_i A)=(\d_{i+3}A\w_i w)\w_{i+2}A$$
holds whenever either side is defined.

\textup{(7)} If $i\leq j-3$ then the equality
$$(x\w_i y)\w_j(z\w_i w)=(x\w_{j-1}z)\w_i(y\w_{j-1}w)$$
holds whenever either side is defined.
\end{definition}

The object of this section is to prove the following result.

\begin{theorem} \label{T9.2}
Let $X$ be a complicial set. Then there is a unique family of wedge operations making~$X$ into a set with complicial identities such that the wedges $x\w_i y$ are $(i+1)$-complicial elements.
\end{theorem}

Since complicial elements in complicial sets are uniquely determined by their faces, the construction of Theorem~\ref{T9.2} is functorial.

The proof of Theorem~\ref{T9.2} uses the method appearing in Street's notes~\cite{B6}. It consists of a sequence of lemmas. In all of these lemmas, $X$~is a complicial set.

\begin{lemma} \label{L9.3}
There is a unique way to construct wedges
$$x\w_i y\in X_{m+1}$$
such that $x\w_i y$ is defined for $x,y\in X_m$ with $\d_i x=\d_{i+1}y$, such that $x\w_i y$ is $(i+1)$-complicial, and such that the conditions of Definition~\ref{D9.1}(1) are satisfied.
\end{lemma}

\begin{proof}
The proof is by induction on~$m$. Let $x,y\in X_m$ be such that $\d_i x=\d_{i+1}y$ and assume that we have constructed wedges with the required properties for pairs of elements with dimensions less than~$m$. We will show that there is an $m$-dimensional $(i+1)$-complicial horn
$$(z_0,\ldots,z_i,z_{i+2},\ldots,z_{m+1})$$
given by
\begin{align*}
&z_j=\d_j x\w_{i-1}\d_j y\quad (0\leq j<i),\\
&z_i=y,\\
&z_{i+2}=x,\\
&z_j=\d_{j-1}x\w_i\d_{j-1}y\quad (i+2<j\leq m+1).
\end{align*}
The conditions on $x\w_i y$ amount to requiring it to be a thin filler of this horn. Since complicial horns in~$X$ have unique thin fillers, this means that there will be a unique wedge $x\w_i y$ with the required properties.

It remains to show that $(z_0,\ldots,z_{m+1})$ really is an $m$-dimensional $(i+1)$-complicial horn. It is straightforward to verify that $\d_k z_j=\d_{j-1}z_k$ for $k<j$ by using Definition~\ref{D9.1}(1) inductively. It also follows from the inductive hypothesis that $z_j$~is $i$-complicial for $j<i$ and that $z_j$~is $(i+1)$-complicial for $j>i+2$. Therefore $(z_0,\ldots,z_{m+1})$ is indeed an $m$-dimensional $(i+1)$-complicial horn. This completes the proof.
\end{proof}

We must now verify the conditions of Definition \ref{D9.1}(2)--(7). Essentially, we need to show that elements of certain forms are wedges. We will write $\im\w_i$ for the subset of~$X$ consisting of the wedges $x\w_i y$ and we will use the following result.

\begin{lemma} \label{L9.4}
Let $A$ be a thin $(m+1)$-dimensional element of~$X$ such that $\d_j A\in\im\w_{i-1}$ for $0\leq j<i$ and such that $\d_j A\in\im\w_i$ for $i+2<j\leq m+1$. Then
$$A=\d_{i+2}A\w_i\d_i A.$$
\end{lemma}

\begin{proof}
There is a wedge $\d_{i+2}A\w_i\d_i A$, because $\d_i\d_{i+2}A=\d_{i+1}\d_i A$, and this wedge is an $(i+1)$-complicial element. The element~$A$ is also $(i+1)$-complicial, because elements of $\im\w_{i-1}$ are $i$-complicial and elements of $\im\w_i$ are $(i+1)$-complicial. Since $(i+1)$-complicial elements in the complicial set~$X$ are uniquely determined by the $(i+1)$-horns which they fill, it suffices to prove that 
$$\d_j A=\d_j(\d_{i+2}A\w_i\d_i A)\quad (j\neq i+1).$$
We use the conditions of Definition \ref{D9.1}(1): for $0\leq j<i$ we have $\d_j A=x_j\w_{i-1}y_j$ such that
\begin{align*}
&x_j=\d_{i+1}(x_j\w_{i-1}y_j)=\d_{i+1}\d_j A=\d_j\d_{i+2}A,\\
&y_j=\d_{i-1}(x_j\w_{i-1}y_j)=\d_{i-1}\d_j A=\d_j\d_i A,
\end{align*}
hence 
$$\d_j A=\d_j\d_{i+2}A\w_{i-1}\d_j\d_i A=\d_j(\d_{i+2}A\w_i\d_i A);$$
we certainly have 
\begin{align*}
&\d_i(\d_{i+2}A\w_i\d_i A)=\d_i A,\\
&\d_{i+2}(\d_{i+2}A\w_i\d_i A)=\d_{i+2} A;
\end{align*}
for $i+2<j\leq m+1$ we have $\d_j A=x_j\w_i y_j$ such that
\begin{align*}
&x_j=\d_{i+2}(x_j\w_i y_j)=\d_{i+2}\d_j A=\d_{j-1}\d_{i+2}A,\\
&y_j=\d_i(x_j\w_i y_j)=\d_i\d_j A=\d_{j-1}\d_i A,
\end{align*}
hence
$$\d_j A=\d_{j-1}\d_{i+2}A\w_i\d_{j-1}\d_i A=\d_j(\d_{i+2}A\w_i\d_i A).$$
This completes the proof.
\end{proof}

When applying this result, we need to show that certain elements are thin. We will use the following result.

\begin{lemma} \label{L9.5}
Let $x\w_i y$ be a wedge in~$X$; then $x\w_i y$ is thin. If also $x$~and~$y$ are thin then $\d_{i+1}(x\w_i y)$ is thin.
\end{lemma}

\begin{proof} 
By construction, $x\w_i y$ is an $(i+1)$-complicial element; therefore $x\w_i y$ is thin. If $x$~and~$y$ are thin then $\d_i(x\w_i y)$ and $\d_{i+2}(x\w_i y)$ are the thin elements $y$~and~$x$, and it follows from Definition \ref{D4.2}(3) that $\d_{i+1}(x\w_i y)$ is thin.
\end{proof}

We will now verify the conditions of Definition \ref{D9.1}(2)--(7).

\begin{lemma} \label{L9.6}
The wedges in~$X$ satisfy the conditions of Definition \ref{D9.1}(2)--(4).
\end{lemma}

\begin{proof}
In each case we proceed by induction on dimension and apply Lemma~\ref{L9.4}. We use Lemma~\ref{L9.5} to recognise thin elements and Definition \ref{D9.1}(1) to compute faces.

(2) Suppose that $x\in X_m$ and and that $0\leq i<m$. The elements $\e_i x$ and $\e_{i+1}x$ are thin because they are degenerate (Definition~\ref{D3.2}). For $0\leq j<i$ it follows from the inductive hypothesis that
$$\d_j\e_i x=\e_{i-1}\d_j x\in\im\w_{i-1},\quad
\d_j\e_{i+1}x=\e_i\d_j x\in\im\w_{i-1};$$
for $i+2<j\leq m+1$ it similarly follows that
$$\d_j\e_i x=\e_i\d_{j-1}x\in\im\w_i,\quad
\d_j\e_{i+1}x=\e_{i+1}\d_{j-1}x\in\im\w_i.$$
By Lemma~\ref{L9.4},
\begin{align*}
&\e_i x=\d_{i+2}\e_i x\w_i\d_i\e_i x=\e_i\d_{i+1}x\w_i x,\\
&\e_{i+1}x=\d_{i+2}\e_{i+1}x\w_i\d_i\e_{i+1}x=x\w_i\e_i\d_i x.
\end{align*}

(3) Let $A$ be an element of the form $b\w_i(y\w_i z)$. Then $A$~is thin because $A$~is a wedge. For $j<i$ it follows from the inductive hypothesis that
$$\d_j A=\d_j b\w_{i-1}(\d_j y\w_{i-1}\d_j z)\in\im\w_i,$$
we certainly have
$$\d_i A=y\w_i z\in\im\w_i,$$
and for $j>i+3$ it follows from the inductive hypothesis that
$$\d_j A=\d_{j-1}b\w_i(\d_{j-2}y\w_i\d_{j-2}z)\in\im\w_{i+1};$$
therefore 
$$A=\d_{i+3}A\w_{i+1}\d_{i+1}A=(\d_{i+2}b\w_i y)\w_{i+1}\d_{i+1}A.$$

(4) Let $A$ be an element of the form $(x\w_i y)\w_{i+1}c$. Then $A$~is thin because $A$~is a wedge. For $j<i$ it follows from the inductive hypothesis that
$$\d_j A=(\d_j x\w_{i-1}\d_j y)\w_i\d_j c\in\im\w_{i-1},$$
we certainly have
$$\d_{i+3}A=x\w_i y\in\im\w_i,$$
and for $j>i+3$ it follows from the inductive hypothesis that
$$\d_j A=(\d_{j-2}x\w_i\d_{j-2}y)\w_{i+1}\d_{j-1}c\in\im\w_i;$$
therefore 
$$A=\d_{i+2}A\w_i\d_i A=\d_{i+2}A\w_i(y\w_i\d_i c).$$
\end{proof}

\begin{lemma} \label{L9.7}
The wedges in~$X$ satisfy the conditions of Definition \ref{D9.1}(5).
\end{lemma}

\begin{proof}
Note that the conditions for the two sides of Definition \ref{D9.1}(5) to be defined are the same, namely $\d_i x=\d_{i+1}y$ and $\d_i y=\d_{i+1}z$. We may therefore assume that the left side is defined. Let
$$A=b\w_i(y\w_i z)$$
with
$$b=x\w_i\d_{i+1}(y\w_i z);$$
we must show that
$$A=(x\w_i y)\w_{i+1}[\d_{i+1}(x\w_i y)\w_i z].$$
But it follows from Definition \ref{D9.1}(3) (already proved in Lemma~\ref{L9.6}) that
$$A
=(\d_{i+2}b\w_i y)\w_{i+1}\d_{i+1}A
=(x\w_i y)\w_{i+1}\d_{i+1}A,$$
and it therefore suffices to prove that 
$$\d_{i+1}A=\d_{i+1}(x\w_i y)\w_i z.$$
This we do by induction in the usual way. We have 
$$\d_{i+1}A=\d_{i+1}[b\w_i(y\w_i z)]$$
such that $b$~and $y\w_i z$ are thin; therefore $\d_{i+1}A$ is thin. For $j<i$ it follows from the inductive hypothesis that $\d_j\d_{i+1}A$ is in $\im\w_{i-1}$ because
$$\d_j\d_{i+1}A=\d_i[\d_j b\w_{i-1}(\d_j y\w_{i-1}\d_j z)]$$
with
$$\d_j b=\d_j x\w_{i-1}\d_i(\d_j y\w_{i-1}\d_j z),$$
and for $j>i+2$ it follows from the inductive hypothesis that $\d_j\d_{i+1}A$ is in $\im\w_i$ because
$$\d_j\d_{i+1}A=\d_{i+1}[\d_j b\w_i(\d_{j-1}y\w_i\d_{j-1}z)]$$
with
$$\d_j b=\d_{j-1}x\w_i\d_{i+1}(\d_{j-1}y\w_i\d_{j-1}z);$$
therefore
$$\d_{i+1}A
=\d_{i+2}\d_{i+1}A\w_i\d_i\d_{i+1}A
=\d_{i+1}(x\w_i y)\w_i z$$
as required.
\end{proof}

\begin{lemma} \label{L9.8}
The wedges in~$X$ satisfy the conditions of Definition \ref{D9.1}(6).
\end{lemma}

\begin{proof}
Again one side is defined if and only if the other is; we may therefore assume that the left side is defined. Let
$$A=\d_{i+2}[(x\w_{i+1}y)\w_{i+1}(y\w_i z)]$$
and let
$$B=A\w_i(w\w_{i+1}\d_i A);$$
we must show that
$$B=(\d_{i+3}A\w_i w)\w_{i+2}A$$
and we apply the usual inductive argument. We write $A_j=\d_j A$, so that
$$A_j=\begin{cases}
\d_{i+1}[(\d_j x\w_i\d_j y)\w_i(\d_j y\w_{i-1}\d_j z)]& (j<i),\\
\d_{i+2}[(\d_{j-1}x\w_{i+1}\d_{j-1}y)\w_{i+1}(\d_{j-1}y\w_i\d_{j-1}z)]& (j>i+3);
\end{cases}$$
in each of these cases $A_j$~is of the same form as~$A$. For $j<i$ it follows from the inductive hypothesis that $\d_j B$ is in $\im\w_{i+1}$ because
$$\d_j B=A_j\w_{i-1}(\d_j w\w_i \d_{i-1}A_j),$$
we certainly have
$$\d_i B=w\w_{i+1}\d_i A\in\im\w_{i+1},$$
and for $j>i+4$ it follows from the inductive hypothesis that $\d_j B$ is in $\im\w_{i+2}$ because
$$\d_j B=A_{j-1}\w_i(\d_{j-2}w\w_{i+1}\d_i A_{j-1}).$$
If we can show that $\d_{i+1}B\in\im\w_{i+1}$ then it will follow that
$$B=\d_{i+4}B\w_{i+2}\d_{i+2}B=(\d_{i+3}A\w_i w)\w_{i+2}A$$
as required.

To show that $\d_{i+1}B\in\im\w_{i+1}$ we again apply the usual inductive argument. By Lemma~\ref{L9.5} $A$~is thin, hence $\d_{i+1}B$ is thin. For $j<i$ it follows from the inductive hypothesis that $\d_j\d_{i+1}B$ is in $\im\w_i$ because
$$\d_j\d_{i+1}B=\d_i[A_j\w_{i-1}(\d_j w\w_i\d_{i-1}A_j)],$$
we certainly have
$$\d_i\d_{i+1}B=\d_i w\w_i\d_i\d_i A\in\im\w_i,$$
and for $j>i+3$ it follows from the inductive hypothesis that $\d_j\d_{i+1}B$ is in $\im\w_{i+1}$ because
$$\d_j\d_{i+1}B=\d_{i+1}[A_j\w_i(\d_{j-1}w\w_{i+1}\d_i A_j)];$$
therefore $\d_{i+1}B\in\im\w_{i+1}$ as required.

This completes the proof.
\end{proof}

\begin{lemma} \label{L9.9}
The wedges in~$X$ satisfy the conditions of Definition \ref{D9.1}(7).
\end{lemma}

\begin{proof}
Again we find that one side is defined if and only if the other is, so we may assume that the left side is defined. Let
$$A=(x\w_i y)\w_j(z\w_i w)$$
with $i\leq j-3$; we will apply the usual inductive argument to show that
$$A=(x\w_{j-1}z)\w_i(y\w_{j-1}w).$$

For $k<i$ it follows from the inductive hypothesis that
$$\d_k A=(\d_k x\w_{i-1}\d_k y)\w_{j-1}(\d_k z\w_{i-1}\d_k w)\in\im\w_{i-1},$$
for $i+2<k<j$ it follows from the inductive hypothesis that
$$\d_k A=(\d_{k-1}x\w_i\d_{k-1}y)\w_{j-1}(\d_{k-1}z\w_i\d_{k-1}w)\in\im\w_i,$$
we certainly have
$$\d_j A=z\w_i w\in\im\w_i$$
and
$$\d_{j+2}A=x\w_i y\in\im\w_i,$$
and for $k>j+2$ it follows from the inductive hypothesis that
$$\d_k A=(\d_{k-2}x\w_i\d_{k-2}y)\w_j(\d_{k-2}z\w_i\d_{k-2}w)\in\im\w_i.$$
If we can show that $\d_{j+1}A\in\im\w_i$ then it will follow that
$$A=\d_{i+2}A\w_i\d_i A=(x\w_{j-1}z)\w_i(y\w_{j-1}w).$$

To show that $\d_{j+1}A\in\im\w_i$ we again apply the usual inductive argument. It follows from Lemma~\ref{L9.5} that $\d_{j+1}A$ is thin. For $k<i$ it follows from the inductive hypothesis that
$$\d_k\d_{j+1}A
=\d_j[(\d_k x\w_{i-1}\d_k y)\w_{j-1}(\d_k z\w_{i-1}\d_k w)]
\in\im\w_{i-1},$$
for $i+2<k<j$ it follows from the inductive hypothesis that
$$\d_k\d_{j+1}A
=\d_j[(\d_{k-1}x\w_i\d_{k-1}y)\w_{j-1}(\d_{k-1}z\w_i\d_{k-1}w)]\in\im\w_i,$$
we certainly have
$$\d_j\d_{j+1}A=\d_{j-1}z\w_i\d_{j-1}w\in\im\w_i$$
and
$$\d_{j+1}\d_{j+1}A=\d_j x\w_i\d_j y\in\im\w_i,$$
and for $k>j+1$ it follows from the inductive hypothesis that
$$\d_k\d_{j+1}A
=\d_{j+1}[(\d_{k-1}x\w_i\d_{k-1}y)\w_j(\d_{k-1}z\w_i\d_{k-1}w)]
\in\im\w_i;$$
therefore $\d_{j+1}A\in\im\w_i$ as required.

This completes the proof.
\end{proof}

It follows from Lemmas \ref{L9.6}--\ref{L9.9} that the wedges of Lemma~\ref{L9.3} satisfy all the axioms for sets with complicial identities. This completes the proof of Theorem~\ref{T9.2}.

\end{document}